\documentclass[dvips, preprint]{imsart}

\RequirePackage[OT1]{fontenc}
\RequirePackage{amsthm,amsmath, amssymb, enumerate}
\RequirePackage[numbers]{natbib}
\RequirePackage[colorlinks,citecolor=blue,urlcolor=blue]{hyperref}
\RequirePackage{hypernat}

\arxiv{math.PR/0000000}

\startlocaldefs
\numberwithin{equation}{section}
\theoremstyle{plain}
\newtheorem{theorem}{Theorem}[section]

\newtheorem{definition}{Definition}[section]
\newtheorem{example}{Example}[section]
\newtheorem{corollary}{Corollary}[section]
\newtheorem{remark}{Remark}[section]

\endlocaldefs

\newcommand{\R}{\mathbb R}
\newcommand{\N}{\mathbb N}

\begin{document}

\begin{frontmatter}
\title{On a connection between Stein characterizations  and Fisher  information}
\runtitle{Stein's density approach  and Fisher information}

\begin{aug}
\author{\fnms{Christophe} \snm{Ley}\thanksref[1]{1}\ead[label=e1]{chrisley@ulb.ac.be}}
\and
\author{\fnms{Yvik} \snm{Swan}\thanksref[2]{2}\ead[label=e2]{yvswan@ulb.ac.be}}

\thankstext[1]{1}{Supported by a Mandat de Charg\'e de recherche from the Fonds National de la Recherche Scientifique, Communaut\'e fran\c{c}aise de Belgique. }
\thankstext[2]{2}{Supported by a Mandat de Charg\'e de recherche from the Fonds National de la Recherche Scientifique, Communaut\'e fran\c{c}aise de Belgique.}
\runauthor{C. Ley and Y. Swan.}

\affiliation{E.C.A.R.E.S. and Universit\'e Libre de Bruxelles}
\address{D\'epartement de Math\'ematique\\
Universit\'e Libre de Bruxelles\\
Campus Plaine -- CP210\\
B-1050 Brussels\\
\printead{e1}, \printead*{e2}}
%
\end{aug}

\begin{abstract} We  generalize  the so-called \emph{density approach} to Stein characterizations of probability distributions.   We prove an elementary  factorization property of the resulting  Stein operator in terms of a \emph{generalized (standardized) score function}. We use this result to  connect Stein characterizations  with information distances such as the  \emph{generalized (standardized) Fisher information}. 

\end{abstract}



\end{frontmatter}


\section{Foreword}
In recent years a number of authors  have noted how Charles Stein's characterization of the Gaussian (see \cite{S72}) and the so-called ``magic factors'' crop up  in matters related to information theory  (see \cite{J04},  \cite{JB04}, \cite{KHJ05}, \cite{CT06} or \cite{BJKM10} and the references therein). The purpose of this   note is to make this connection explicit.

\section{Results} \label{sub:st_inf}

 We consider   densities $p: \R\to \R^+$  whose support 
is an interval $S := S_p$ with closure  $\bar S= [a, b]$, for some $-\infty\le a < b \le \infty$.   Among these we denote by  $\mathcal G$  the collection of densities which are (strongly) differentiable at every point in the interior of their support. 

\begin{definition} \label{def:class}Fix $p \in \mathcal G$ with support $S$ and define  $\mathcal{F}(p)$ the  collection of  \emph{test functions} $f: \R\to \R$ such that the mapping $x \mapsto  f(x)p(x)$ 
  is bounded on $\R$ and  strongly differentiable on the interior of $S$.
\end{definition}

Take a real bounded function $h$ with support $S$, and suppose that $h$ is (strongly) differentiable on the interior of $S$. Then $h$ can be written as $\tilde{h} \mathbb{I}_S$ with $\tilde h$ a differentiable continuation of $h$ on $\R$.  In the sequel we will write $\partial_y h(y)|_{y=x}$ for the differential in the sense of distributions of $h$ evaluated at $x$, so that  $\partial_y h(y)|_{y=x} = (\tilde{h})'(x)\mathbb{I}_S(x)+\tilde{h}(x)\left(\delta_{\{x=a\}}-\delta_{\{x=b\}}\right)$ where $\delta$ represents a Dirac delta.

\begin{definition} \label{def:operator} Let $\R^\star$ be   the collection of all functions $f: \R \to \R$. We define the  \emph{(location-based) Stein operator}   as the   operator $\mathcal{T} : \R^\star \times \mathcal G \to \R^\star  : (f, p) \mapsto  \mathcal T (f, p)$
 such that 
\begin{equation}\label{eq:st_op} \mathcal T (f, p): \R\rightarrow\R:x \mapsto \frac{\left.\partial_y(f(y)p(y))\right|_{y=x}}{p(x)}\end{equation}
for all $f$ for which the differential (in the sense of distributions) exists.
\end{definition}

\begin{remark}The terminology ``location-based'' Stein operator  is inherited from  our parametric approach to Stein characterizations (see \cite{LS11}), where a much more general characterization result is proposed. 
\end{remark} 

To avoid ambiguities related to division by 0, throughout this paper we  adopt the convention that, whenever an expression involves the division by an indicator function $\mathbb{I}_A$ for some measurable set $A$,  we are multiplying the expression by the said indicator function. This convention ensures that for $p\in \mathcal G$ and $f \in \mathcal{F}(p)$ and for any continuous random variable $X$, the quantity  $\mathcal T (f, p)(X)$ is well-defined.  We further draw the reader's attention to the fact that, in particular,   ratios   $p(x) / p(x)$ do not necessarily  simplify to 1.

\begin{example}\label{ex:operators}
It is perhaps informative to see how Definitions~\ref{def:class} and \ref{def:operator} spell out for different explicit choices of target densities.  
\begin{enumerate}
\item  If $p=\phi$,  the standard Gaussian, then $\mathcal{F}(\phi)$ contains the set  of all differentiable bounded functions and 
$$\mathcal T(f, \phi)(x) = f'(x) - xf(x),$$
which is  Stein's well-known operator for characterizing the Gaussian. 
\item If $p(x)=  e^{-  x}Ê\mathbb{I}_{[0, \infty)}(x)$, the exponential $Exp$, then (abusing  notations) $\mathcal{F}(Exp)$ contains the set  of all differentiable bounded functions and  
$$\mathcal T(f, Exp)(x) = \left(f'(x) -  f(x) + f(x)\delta_{\{x=0\}}\right)\mathbb{I}_{[0,\infty)}(x).$$
\item If $p(x) = \mathbb I_{[0,1]}(x)$,  the standard uniform $U(0, 1)$, then    $\mathcal{F}(U(0, 1))$ contains the set  of all differentiable bounded functions and
$$\mathcal T(f, U(0, 1) )(x) = \left(f'(x)  +f(x)(\delta_{\{x=0\}}-\delta_{\{x=1\}})\right)\mathbb{I}_{[0,1]}(x).$$
\item If $p(x) = \frac{1}{2\pi}\sqrt{4-x^2}\mathbb{I}_{(-2, 2)}(x)$,  Wigner's semicircle  law $SC$, then   $\mathcal{F}(SC)$ contains the set  of all  functions of the form $f(x) = f_0(x)({4-x^2})$ for some bounded differentiable $f_0$ and, for these $f$, the operator becomes  
$$\mathcal T(f, SC )(x) = \left((4-x^2)f_0'(x) -3xf_0(x)\right)\mathbb{I}_{(-2,2)}(x).$$

\item If $f(x) =   \frac{1}{\pi\sqrt{x(1-x)}} \mathbb{I}_{(0, 1)}(x)$, the arcsine distribution $AS$, then $\mathcal{F}(AS)$ contains the collection of all functions of the form $f(x) = f_0(x) \sqrt{x(1-x)}$ for some bounded differentiable  $f_0$ and, for these $f$, the operator becomes 
$$\mathcal T(f, AS )(x) = \sqrt{x(1-x)} f_0'(x) \mathbb{I}_{(0,1)}(x).$$

\item If $p(x)$  is a member of Pearson's family of distributions and thus satisfies 
$$(s(x)p(x))' = \tau(x) p(x) $$
for $\tau$ a polynomial of exact degree one and $s$ a polynomial of degree at most two, then, abusing notations one last time, we easily see that $\mathcal{F}(P(s, \tau))$ contains the set of all functions of the form 
$f(x) = f_0(x) s(x)$   for   $f_0$  bounded,  differentiable such that $f(a^+) = f(b^-)=0$  and, for these $f$, the operator becomes 

$$\mathcal{T}(f, P(s, \tau))(x) = (s(x)f_0'(x) + \tau(x) f_0(x))\mathbb I_{S}(x).$$
\end{enumerate}
The first three operators are well-known and can be found, for instance, in \cite{SDHR04}. The fourth example can be found in \cite{GT07}. The last example comes from  \cite{S01}. 
\end{example} 
 
We are now ready to state and prove our first main result. 

\begin{theorem}[Density approach]\label{theo1} Let $p \in \mathcal G$ with support $S$, and take $Z \sim p$. Let $\mathcal{F}(p)$  be as in Definition \ref{def:class} and $\mathcal T$ as in Definition \ref{def:operator}. Let $X$ be a real-valued continuous random variable. 
\begin{enumerate}[(1)]
\item If $X \stackrel{\mathcal L}{=}ÊZ$ then ${\rm E}\left[\mathcal{T}(f, p)(X) \right]=0$ for all $f \in \mathcal{F}(p)$.
\item  If   ${\rm E}[\mathcal{T}(f, p)(X)]=0$ for all $f\in\mathcal{F}(p)$, then $X\,|\,X\in S\stackrel{\mathcal{L}}{=} Z$.
 \end{enumerate}
\end{theorem}
 
\begin{proof}
To see (1), note that the hypotheses on $f$ and $p$ guarantee that we have
$ {\rm E}\left[\mathcal{T}(f, p)(Z) \right]        = \left[f(y)p(y)\right]_a^b + f(a^+)p(a^+)-f(b^-)p(b^-) = 0.$
To see (2), consider for $z\in \R$ the functions $f_z^p$ defined through
$$ f_z^p: \R\to \R: x \mapsto \frac{1}{p(x)} \int_{a}^x l_z(u) p(u) du$$
 with $l_z(u):= ({\mathbb{I}}_{(- \infty, z]}(u) - {\rm P}_p(X \le z))\mathbb{I}_S(u)$ and  ${\rm P}_p(X\le z):=\int_{-\infty}^z p(u)du$.  
 Clearly $f_z^p \in \mathcal{F}(p)$ for all $z$. Moreover we have $\partial_y (f_z^p(y)p(y))|_{y=x} = l_z(x)p(x)$ since $\int_a^cl_z(u)p(u)du=0$ for $c=a$ and $c=b$. Therefore $f_z^p$ satisfies, for all $z$,  the so-called \emph{Stein equation}
\begin{equation}\label{eq:st}\mathcal{T}(f_z^p, p)(x) = l_z(x).\end{equation}
Hence we can  use ${\rm E}\left[\mathcal{T}(f_z^p, p)(X)\right] =0$ to deduce that  $ {\rm P}(X \in (-\infty, z] \cap S) =  {\rm P}(Z\le z)  {\rm P}(X \in S)$
for all $z$, whence the claim. 
%
 \end{proof}

Theorem \ref{theo1} encompasses Proposition 4 in \cite{SDHR04}  and  Theorem 1 in \cite{S01} and  is easily shown to contain  many of  the other better known  Stein characterizations (such as the characterization of the semi-circular in~\cite{GT07}). We draw  the reader's attention to the fact that our way of writing the Stein operator \eqref{eq:st_op} alsoÊ shows that  all  Stein equations of the form \eqref{eq:st} (that is, most such equations from the literature) can be solved by simple integration. Also, the form of our operators leads directly to our second main result.

\begin{theorem}[Factorization Theorem of Stein Operators] \label{lemma:facto} Let $p$ and $q$ be  probability density functions in $\mathcal G$ sharing support $S$.  For all $f \in \mathcal{F}(p)\cap \mathcal{F}(q)$, we have
$$\mathcal{T}(f, p)(x) = \mathcal{T}(f, q)(x)+ f (x) r(p, q)(x),$$
with 
$$r(p, q)(x) :=    \frac{p'(x)}{p(x)} - \frac{q'(x)}{q(x)} +(\delta_{\{x=a\}}-\delta_{\{x=b\}})\mathbb{I}_S(x).$$
\end{theorem}

\begin{proof}
The restriction on the support of $q$ guarantees that we have 
$f(y)p(y) = f(y)q(y)p(y)/q(y) $
 for any real-valued function $f$.
We can therefore write
 \begin{align*} \mathcal{T}(f, p)(x)  
 						      &  = \frac{\left.\partial_y(f(y)q(y) p(y)/q(y))\right|_{y=x}}{p(x)}   \\
 						      & = \frac{\left.\partial_y(f(y)q(y))\right|_{y=x}}{p(x)} \frac{p(x)}{q(x)}+ f(x)q(x)\frac{\left.\partial_y( p(y)/q(y))\right|_{y=x}}{p(x)}\\
						      & = \mathcal{T}(f, q)(x) + f(x)\frac{q(x)}{p(x)}\left.\partial_y( p(y)/q(y))\right|_{y=x}.
						      \end{align*}
The claim follows.
\end{proof}
 
Note that, whenever $S=\R$ or $S$ is an open interval, $r(p,q)$ simplifies to $p'/p-q'/q$. Now, let $l$ be a real-valued function. In the sequel we will write  ${\rm E}_p[l(X)] := \int_\R l(x) p(x) dx.$ Our next and final main result is immediate and hence its proof is left to the reader. 
 \begin{theorem}[Stein's method and information distances] \label{th:steinentropy}    
Let $p$ and $q$ be  probability density functions in $\mathcal G$ sharing support $S$. Let $l$ be a real-valued function such that ${\rm E}_p[ l(X)]$ and ${\rm E}_q[ l(X)]$ exist. Define $f_l^p$ to be the solution of the Stein equation 
\begin{equation}\label{eq:steineq22}\mathcal{T}(f, p)(x) =   (l(x)- {\rm E}_p[l(X)])\mathbb{I}_{S}(x)\end{equation}
 and suppose that $f_l^p \in \mathcal{F}(q)$. Then 
\begin{equation} \label{eq:fundeq}  {\rm E}_q[l(X)] -{\rm E}_p[l(X)] =  
{\rm E}_q [{f}_l^p(X) r(p , q)(X)].
\end{equation}
 
\end{theorem}

 Whenever $p$ is well-behaved, the solutions to \eqref{eq:steineq22} are of the well-known form 
\begin{equation}\label{eq:steineq_sol2}f_l^p : \R \to \R : x \mapsto \frac{1}{p(x)}\int_{a}^x  (l(u)- {\rm E}_p[l(X)]) p(u)du.\end{equation}
In cases such as the SC or the AS, the form of this solution (expressed in terms of $f_0$ instead of $f$) is slightly different but easily provided, see Example \ref{ex:operators} or equations (18) and (19) in Proposition 1 of \cite{S01}.  

In all explicit  instances covered in Example \ref{ex:operators},  the condition that $f_l^p \in \mathcal{F}(q)$ is trivially verified (see page 4 in \cite{CS05} for the Gaussian). Under moment conditions on $p$, Schoutens shows  that  members of the Pearson family satisfy this assumption as well (see \cite{S01}, Lemma 1).

\section{Application}\label{sec:app}

Applying H\"older's inequality to \eqref{eq:fundeq} shows that, under the same conditions, 
\begin{equation} \label{eq:fundeq2} |{\rm E}_q[l(X)] - {\rm E}_p[l(X)] | \le \kappa_l^p\sqrt{{\rm E}_q [(r(p, q)(X))^2]},\end{equation}
with 
$$\kappa_l^p = \sqrt{{\rm E}_q [(f_l^p(X))^2]}.$$

Equation \eqref{eq:fundeq2} provides  some form of universal bound on differences of expectations in terms of what can be likened to  a  \emph{generalized (standardized) Fisher information distance} $$\mathcal{J}(p, q) = {\rm E}_q[(r(p,q)(X))^2]$$ 
(the terminology and notations are borrowed from \cite{BJKM10}). Note how, for instance, taking $p= \phi$ the standard Gaussian density yields the \emph{Fisher information distance} studied, e.g., in \cite{JB04}. 

Theorem \ref{th:steinentropy}  also   provides a bound on any probability metric which can be written as 
\begin{equation}\label{eq:hdizst}d_{\mathcal H} (p, q) = \sup_{h \in \mathcal{H}} \left|{\rm E}_q[h(X)] - {\rm E}_p[h(X)] \right|\end{equation}
for some class of functions $\mathcal{H}$. The \emph{Kolmogorov, Wasserstein} and \emph{total variation} distances, to cite but these, can all  be written in this form.  

Specifying the target as well as the class $\mathcal H$  yields the following immediate corollaries.

\begin{corollary} \label{cor:shimizu} Let  $p$ and $q$ be probability densities with  support $S\subseteq\R$ satisfying the hypotheses in Theorem \ref{th:steinentropy}. Then there exist  constants $\kappa_1:=\kappa_1(p,q)$ and $\kappa_2:=\kappa_2(p,q)$  such that 
\begin{equation*} \int | p(u) - q(u) | du \le \kappa_1 \sqrt{\mathcal{J}(p, q)}\end{equation*}
and
\begin{equation*} \sup_{x\in \R} | p(x) - q(x) |  \le \kappa_2 \sqrt{\mathcal{J}(p, q)}.\end{equation*}

\end{corollary}

\begin{proof}
Take  $l(u) = \mathbb{I}_{\{p(u) \le q(u)\}} - \mathbb{I}_{\{p(u)\ge q(u)\}}.$ Using \eqref{eq:fundeq} with this choice of $l$ and applying H\"older's inequality,   one readily sees that  there exists a constant $\kappa_1>0$ such that 
\begin{equation*} \label{eq:fisher_hell}  \int |p(x) - q(x)| dx \le    \kappa_1 \sqrt{\mathcal{J}(p, q)}\end{equation*}
where $\kappa_1 = \sqrt{{\rm E}_q[(f_l^p(X))^2]}$. 

Regarding the second inequality first note that, whenever $x\in S^c$, $|p(x)-q(x)|=0$, hence we can concentrate on the supremum over the support $S$. Now choose  $l(u) = \delta_{\{x=u\}}$ the Dirac delta function in $x \in S$. For this choice of $l$ we  obtain after some computations
$$|q(x) -p(x)| \le p(x)\sqrt{ {\rm E}_q \left[\left(\mathbb{I}_{[x, b)}(X) - P(X)\right)^2/(p(X))^2\right]} \sqrt{\mathcal{J}(p, q)}, $$
where $P$ is the cumulative distribution function of the density $p$ (for which evidently $P(a)=0$). Taking the supremum yields the second constant $\kappa_2$.

\end{proof}

We conclude this paper with a computation of bounds on the constants $\kappa_1$ and $\kappa_2$ for various examples. While these   are   somewhat related  to the so-called ``magic factors'' appearing in the literature on Stein's method,  the technique we employ to bound them    is different and, we believe, of independent interest.
 To the best of our knowledge,  such  bounds were first obtained in   \cite{S75} for Gaussian target only. Shimizu's results were later  improved and extended in \cite{J04} and \cite{JB04}.    We recover in Corollary~\ref{cor2_1} below the best known values for $\kappa_1$ and our bound for $\kappa_2$ yields a significant improvement.  We stress the fact that the  results available in the literature  only concern a Gaussian target, whereas  our approach allows  to obtain such relationships for  virtually any target distribution. Further explorations of the   consequences of Theorem~\ref{th:steinentropy}  also show that it is possible to relate Stein characterizations with other (pseudo-)metrics than those of the form \eqref{eq:hdizst},   such as, e.g., \emph{Kullback-Leibler divergence} or  \emph{relative entropy} (see \cite{J04}). 

\begin{corollary} \label{cor2_1} $\mbox{}$\begin{enumerate}[1.]
\item  If $p$ is the exponential distribution with rate 1,  then $\kappa_1 \le 1$.
\item  If  $p=\phi$ is the standard normal distribution, then $\kappa_1 \le \sqrt2$.
\item   If  $p$ is proportional to  $e^{-x^4/12}$,  then $\kappa_1 \le \sqrt{2\sqrt2}$.
\end{enumerate}
In all three cases we have $\kappa_2 \le 1$. 

\end{corollary}

\begin{proof}[Proof of the constants $\kappa_1$]

Take  $l(u) = \mathbb{I}_{\{p(u) \le q(u)\}} - \mathbb{I}_{\{p(u)\ge q(u)\}}.$ Using  \eqref{eq:steineq_sol2} and the fact that  $\int_a^b(l(u) - {\rm E}_p[l(X)])p(u)du=0$, we obtain that 
\begin{align*} f_l^p(x)  & = -\frac{1}{p(x)}\int_x^b (l(u) - {\rm E}_p[l(X)] )p(u) du \\
				& =- \frac{2}{p(x)} \int_x^b \left( \mathbb{I}_{\{p(u)\le q(u)\}}- {\rm P}_p(p(X) \le q(X) \right)p(u)du\\
				& =: \frac{2}{p(x)} \int_x^b h(u)p(u)du,\end{align*}
where ${\rm P}_p(X\in A)=\int_A p(u)du$ for some set $A$. Let $p(x)= e^{-x}\mathbb{I}_{[0,\infty)}(x)$, the  density of an exponential-1 random variable (in other words, $a=0$ and $b=\infty$). Recall that, in this case, the support of $f_l^p$ is a subset of $\R^+$.  Then we can write 
\begin{align*} \kappa_1:= {\rm E}_q[(f_l^p(X))^2] & = 4 \int_{0}^\infty q(x) e^{2x}\left( \int_{x}^\infty  h(u)e^{-u} du\right)^2 dx \\
					& \le 4\int_{0}^\infty  q(x)  e^{2x} \left(\int_{x}^\infty  h^2(u)e^{-2u} du\right) dx \\
					& \le \frac{4}{2} \int_{0}^\infty  q(x)  e^{2x} \left(\int_{2x}^\infty  h^2\left(\frac{u}{2}\right)e^{-u} du\right) dx,
										\end{align*}
where the first inequality follows from Jensen and the second inequality from a simple change of variables. Applying H\"older's inequality and again changing variables in the above yields 
\begin{align*} \kappa_1 & \le \frac{4}{2} \sqrt{\int_{0}^\infty  q(x)dx}\sqrt{\int_0^\infty  q(x)e^{4x} \left(\int_{2x}^\infty  h^2\left(\frac{u}{2}\right)e^{-u} du\right)^2 dx}\\
					& \le \frac{4}{2^{1+\frac{1}{2}}} \left(\int_0^\infty  q(x)e^{4x} \left(\int_{4x}^\infty  h^4\left(\frac{u}{4}\right)e^{-u} du\right) dx\right)^{1/2},
										\end{align*}
where $\int_{0}^\infty  q(x)dx=1$ by our assumption that $p$ and $q$ share the same support. Iterating this procedure $m\in \N$ times, we obtain 
$$\kappa_1 \le \frac{4}{2^{M(m)}} \left(\int_0^\infty  q(x)e^{2^{m+1}x} \left(\int_{2^{m+1}x}^\infty  h^{2^{m+1}}\left(\frac{u}{2^{m+1}}\right)e^{-u} du\right) dx\right)^{1/2^{m}},$$
where $M(m) = 1+\frac{1}{2}+\ldots + \frac{1}{2^m}$. Now note that, for each $m\ge0$, we have  $0 \le h^{2^{m+1}}(u/2^{m+1}) \le 1$. Hence 
$$\int_{2^{m+1}x}^\infty  h^{2^{m+1}}\left(\frac{u}{2^{m+1}}\right)e^{-u} du \le e^{-2^{m+1}x}.$$
Since $M(m) \to 2$, the result follows.

If the support of $p$ (and hence also of $q$) is the real line, we use similarly as above the identity $\int_{-\infty}^\infty (l(u) - {\rm E}_p[l(X)] )p(u) du = 0$ to write, equivalently,
$$ f_l^p(x) = \frac{2}{p(x)} \int_{x}^\infty h(u)p(u)du= - \frac{2}{p(x)} \int_{-\infty}^x h(u)p(u)du.$$
This yields
\begin{align*} {\rm E}_q[(f_l^p(X))^2] & = 4 \int_{-\infty}^\infty q(x)\left( \frac{1}{p(x)}\int_{x}^\infty  h(u)p(u) du\right)^2 dx \\
					& = 4 \int_{-\infty}^0 q(x) \left( \frac{1}{p(x)} \int_{-\infty}^x  h(u)p(u) du\right)^2 dx \\
					& \quad \quad \quad + 4\int_{0}^\infty  q(x) \left( \frac{1}{p(x)} \int_{x}^\infty  h(u)p(u) du\right)^2 dx.\end{align*}
Setting $p(x) = (2\pi)^{-1/2}e^{-x^2/2}$ we get by Jensen's inequality
\begin{align*} {\rm E}_q[(f_l^p(X))^2] & \le  4 \int_{-\infty}^0 q(x)  \left(e^{x^2} \int_{-\infty}^x  h^2(u)e^{-u^2} du\right) dx \\
					& \quad \quad \quad + 4\int_{0}^\infty  q(x) \left( e^{x^2}  \int_{x}^\infty  h^2(u)e^{-u^2}  du\right) dx =:I^-+I^+.\end{align*}
Both integrals above can be tackled in the same way as for the exponential distribution. Consider, for instance, $I^-$ for which we can write (thanks to a simple change of variables)
\begin{align*}  I^- & = 4 \int_{-\infty}^0 q(x)  \left(e^{x^2} \int_{-\infty}^x  h^2(u)e^{-u^2} du\right) dx \\	
		        & = \frac{4}{\sqrt2}  \int_{-\infty}^0 q(x)  \left(e^{x^2} \int_{-\infty}^{\sqrt2 x}   h^2(u/\sqrt2)e^{-u^2/2} du\right) dx. \end{align*}
Now apply  H\"older's inequality to get 		        
\begin{align*} 		     I^-   & \le \frac{4\sqrt{p}}{\sqrt2}   \sqrt{ \int_{-\infty}^0 q(x)  \left(e^{x^2} \int_{-\infty}^{\sqrt2 x}   h^2(u/\sqrt2)e^{-u^2/2} du\right)^2dx } \\
					& \le  \frac{4\sqrt{p}}{\sqrt2}   \sqrt{ \int_{-\infty}^0 q(x)  \left(e^{2x^2} \int_{-\infty}^{\sqrt2 x}   h^4(u/\sqrt2)e^{-u^2} du\right)dx }, \end{align*}
where $p={\rm P}_q(X<0)$. Changing variables once more yields 
$$I^-\le I_1^-$$
with 
\begin{align*} 	I_1^- =	  \frac{4 p^{\frac{1}{2}}}{2^{\frac{1}{2}+\frac{1}{4}}}  \left({ \int_{-\infty}^0 q(x)  \left(e^{2x^2} \int_{-\infty}^{(\sqrt2)^2 x}   h^4\left(\frac{u}{(\sqrt2)^2}\right)e^{-u^2/2} du\right)dx }\right)^{\frac{1}{2}}. \end{align*}
Iterating this procedure $m\in \mathbb{N}$ times we deduce 
$$I^-\le I_1^-\le\ldots \le  I_m^-$$
with $I_m^-$ given by 
\begin{align*} & \frac{4p^{N(m)} }{2^{N(m+1)}}  \left(\int_{-\infty}^0 q(x)  \left(e^{2^mx^2} \int_{-\infty}^{(\sqrt2)^{m+1} x}   h^{2^{m+1}}\left(\frac{u}{(\sqrt2)^{m+1}}\right)e^{-u^2/2} du\right)dx\right)^{\frac{1}{2^m}}\end{align*}				
where we set $N(m)=\frac{1}{2}+\frac{1}{4}+\ldots+\frac{1}{2^m}(=M(m)-1)$. For every $m$ we have $0\le h^{2^{m+1}}(u/(\sqrt2)^{m+1}) \le 1$ and
$$\int_{-\infty}^0 q(x)  \left(e^{2^mx^2} \int_{-\infty}^{(\sqrt2)^{m+1} x} e^{-u^2/2} du\right)dx \le {\rm P}_q(X<0) \sqrt{\frac\pi2}.$$
Therefore 
$$I_m^- \le \frac{4}{2^{N(m+1)}}({\rm P}_q(X<0))^{N(m)}\left({\rm P}_q(X<0) \sqrt{\frac\pi2}\right)^{1/2^m}.$$
Since $N(m) \to 1$ as $m\to \infty$, we conclude 
$$I^- \le 2\,{\rm P}_q(X<0).$$
One can similarly show that $I^+ \le 2\,{\rm P}_q(X>0)$, and the result follows. 

The computations for densities proportional to $e^{-x^4/12}$ are similar and are left to the reader. 
 
\end{proof}

\begin{proof}[Proof of the constants $\kappa_2$]
Let $p(x)=e^{-x}\mathbb{I}_{[0,\infty)}(x)$, which readily implies $P(x)=(1-e^{-x})\mathbb{I}_{[0,\infty)}(x)$. This leads to
\begin{align*}
& {\rm E}_q\left[\left(\mathbb{I}_{[x, \infty)}(X) - P(X)\right)^2/(p(X))^2\right] \\
& = \int_0^\infty q(y)e^{2y}(\mathbb{I}_{[x, \infty)}(y)-1+e^{-y})^2dy\\
&= \int_0^x q(y)e^{2y}(1-e^{-y})^2dy +  \int_x^\infty q(y)e^{2y}e^{-2y}dy\\
&= \int_0^x q(y)e^{2y}(1-2e^{-y})dy +\int_0^x q(y)dy+  \int_x^\infty q(y)dy\\
&\leq 1+ e^{2x}(1-2e^{-x}){\rm P}_q(X\leq x),
\end{align*}
since $e^{2y}(1-2e^{-y})$ is a monotone increasing function on $\R^+$. This immediately yields
$$\kappa_2\leq \sup_{x\geq0}\left(e^{-x}\sqrt{1+ e^{2x}(1-2e^{-x}){\rm P}_q(X\leq x)}\right),$$
a quantity which can be bounded by 1.

Now let $p(x)=(2\pi)^{-1/2}e^{-x^2}$ and $P(x)=\Phi(x)$, the cumulative distribution function of the standard normal distribution. Similarly as for the exponential, we have
\begin{align*}
{\rm E}_q[\left(\mathbb{I}_{[x, \infty)}(X) - P(X)\right)^2/(p(X))^2]&\\
& \hspace{-4cm} = 2\pi\int_{-\infty}^\infty q(y)e^{y^2}(\mathbb{I}_{[x, \infty)}(y)-\Phi(y))^2dy\\
& \hspace{-4cm}= 2\pi\int_{-\infty}^x q(y)e^{y^2}(\Phi(y))^2dy +2\pi  \int_x^\infty q(y)e^{y^2}(1-\Phi(y))^2dy\\
& \hspace{-4cm}\leq 2\pi e^{x^2}(\Phi(x))^2\int_{-\infty}^x q(y)dy +2\pi e^{x^2}(1-\Phi(x))^2\int_x^\infty q(y)dy\\
& \hspace{-4cm}= 2\pi e^{x^2}(\Phi(x))^2 + 2\pi e^{x^2}(1-2\Phi(x)){\rm P}_q(X\geq x).
\end{align*}
This again directly leads to
\begin{align*}
\kappa_2&\leq \sup_{x\in\R}\left((2\pi)^{-1/2}e^{-x^2/2}\sqrt{2\pi e^{x^2}((\Phi(x))^2 +  (1-2\Phi(x)){\rm P}_q(X\geq x)}\right)\\
&=\sup_{x\in\R}\left(\sqrt{(\Phi(x))^2 +  (1-2\Phi(x)){\rm P}_q(X\geq x)}\right),
\end{align*}
a quantity which can be shown to equal 1.

The computations for densities proportional to $e^{-x^4/12}$ are similar and are left to the reader. 
\end{proof}


\end{document}